\documentclass[12pt]{amsart}

\usepackage{amssymb}
\usepackage{bm}
\usepackage{graphicx}
\graphicspath{{./images/}} 
\usepackage{psfrag}
\usepackage{color}
\usepackage{url}
\usepackage{algpseudocode}
\usepackage{fancyhdr}
\usepackage{xy}
\input xy
\xyoption{all}

\newtheorem{theorem}{Theorem}[section]
\newtheorem{proposition}[theorem]{Proposition}

\newtheorem{lemma}[theorem]{Lemma}
\newtheorem{cor}[theorem]{Corollary}

\theoremstyle{definition}
\newtheorem{definition}[theorem]{Definition}
\newtheorem{example}[theorem]{Example}

\numberwithin{equation}{section}

\newcommand\nn{\mathbb{N}}
\newcommand\pp{\mathbb{P}}
\newcommand\qq{\mathbb{Q}}
\newcommand\rr{\mathbb{R}}
\newcommand\zz{\mathbb{Z}}

\newcommand\gp{\mathsf{gp}}

\keywords{Puiseux monoids, factorization theory, factorization invariants, system of sets of lengths, set of distances, realization theorem, catenary degree}


\begin{document}

\mbox{}
\title{Atomicity and density of Puiseux monoids}

\author{Maria Bras-Amoros}
\address{Departament d'Enginyeria Inform{\`a}tica i Matem{\`a}tiques \\Universitat Rovira i Virgili \\ Avinguda dels Pa{\"i}sos Catalans 26 \\ E-43007 Tarragona \\ Spain}
\email{maria.bras@urv.cat}

\author{Marly Gotti}
\address{Research and Development \\ Biogen \\ Cambridge \\ MA 02142 \\ USA}
\email{marly.cormar@biogen.com}

\date{\today}

\begin{abstract}
	A Puiseux monoid is a submonoid of $(\qq,+)$ consisting of nonnegative rational numbers. Although the operation of addition is continuous with respect to the standard topology, the set of irreducibles of a Puiseux monoid is, in general, difficult to describe. In this paper, we use topological density to understand how much a Puiseux monoid, as well as its set of irreducibles, spread through $\rr_{\ge 0}$. First, we separate Puiseux monoids according to their density in $\rr_{\ge 0}$, and we characterize monoids in each of these classes in terms of generating sets and sets of irreducibles. Then we study the density of the difference group, the root closure, and the conductor semigroup of a Puiseux monoid. Finally, we prove that every Puiseux monoid generated by a strictly increasing sequence of rationals is nowhere dense in $\rr_{\ge 0}$ and has empty conductor.
\end{abstract}
\medskip

\maketitle

\section{Introduction}
\label{sec:intro}

A Puiseux monoid is an additive submonoid of $(\qq_{\ge 0},+)$. The first significant appearance of Puiseux monoids in commutative algebra seems to date back to the 1970s, when A. Grams used them in~\cite{aG74} to disprove P. Cohn's conjecture (see~\cite{pC68}) that every atomic integral domain satisfies the ACCP (i.e., every ascending chain of principal ideals eventually stabilizes). However, it was not until recently that Puiseux monoids became the focus of significant attention in factorization theory because of their rich and complex atomic structure. The first systematic study of Puiseux monoids appeared in~\cite{fG17}, and since then they have been present in the semigroup and factorization theory literature (see, for instance,~\cite{CGG20a} and~\cite{fG18b}). Recent applications of Puiseux monoids to numerical semigroups and commutative algebra can be found in~\cite{GS18} and~\cite{CG19}, respectively.

In general, the atomic structure of a Puiseux monoid can be significantly complex. Puiseux monoids range from antimatter monoids (i.e., monoids without atoms) such as $\langle 1/2^n \mid n \in \nn \rangle$ to atomic monoids whose sets of atoms are dense in $\rr_{\ge 0}$ (see Example~\ref{ex:atomically dense PM}). Even though sufficient conditions for atomicity have been found (see~\cite[Theorem~5.5]{GG18} and~\cite[Proposition~4.5]{fG19}), there is no characterization of atomic Puiseux monoids in terms of their generating sets. In this paper we study how much Puiseux monoids and, in particular, their sets of atoms, can spread through $\rr_{\ge 0}$, hoping our study can contribute towards the understanding of their atomic structure.

In Section~\ref{sec:general facts of DPM}, we subclassify Puiseux monoids according to how much they spread throughout $\rr_{\ge 0}$: Puiseux monoids that are dense (in $\rr_{\ge 0}$), Puiseux monoids that are eventually dense (i.e., dense in $[r,\infty)$ for some $r > 0$), Puiseux monoids that are somewhere dense (i.e., dense in some finite interval), and Puiseux monoids that are nowhere dense (in $\rr_{\ge 0}$). We characterize members of each of these four classes in terms of their sets of atoms and generating sets. In this section we also exhibit an atomic Puiseux monoid whose set of atoms is dense in~$\rr_{\ge 0}$.

In the first part of Section~\ref{sec:increasing PM}, we argue that the difference group (resp., the root closure) of a Puiseux monoid $M$ is dense in $\rr$ (resp., in $\rr_{\ge 0}$) provided that $M$ is not finitely generated. Then we fully describe the density of Puiseux monoids with nonempty conductor: such Puiseux monoids are nowhere dense if and only if they are finitely generated. In the second part of Section~\ref{sec:increasing PM}, we restrict our attention to increasing Puiseux monoids. A submonoid of $(\rr_{\ge 0},+)$ is called increasing provided that it can be generated by an increasing sequence. Increasing monoids were first studied in~\cite{GG18} in the context of Puiseux monoids, and then they were investigated in~\cite{mBA19,mBA20}. Increasing monoids are always atomic~\cite[Proposition~4.5]{fG19}. We conclude this paper proving that every non-finitely generated increasing Puiseux monoid is nowhere dense and has empty conductor.

\bigskip
\section{Preliminary}
\label{sec:Background and Notation}

\subsection{General Notation} In this section, we review most of the notation and terminology we shall be using later. The interested reader can consult~\cite{pG01} for background material on commutative semigroups and~\cite{GH06} for extensive information on factorization theory of atomic monoids. The symbol $\mathbb{N}$ (resp., $\mathbb{N}_0$) denotes the set of positive integers (resp., nonnegative integers), while $\pp$ denotes the set of primes. For $r \in \rr$ and $S \subseteq \rr$, we let $S_{\ge r}$ denote the set $\{s \in S \mid s \ge r\}$ and, in a similar manner, we shall use the notation $S_{> r}$. If $q \in \qq_{> 0}$, then we call the unique $a,b \in \nn$ such that $q = a/b$ and $\gcd(a,b)=1$ the \emph{numerator} and \emph{denominator} of $q$ and denote them by $\mathsf{n}(q)$ and $\mathsf{d}(q)$, respectively. For each subset $S$ of $\qq_{>0}$, we call the sets $\mathsf{n}(S) = \{\mathsf{n}(q) \mid q \in S\}$ and $\mathsf{d}(S) = \{\mathsf{d}(q) \mid q \in S\}$ the \emph{numerator set} and \emph{denominator set} of $S$, respectively.

\medskip
\subsection{Monoids}

Every time the term ``monoid" is mentioned here, we tacitly assume that the monoid in question is commutative and cancellative. Unless we specify otherwise, we use additive notation on any monoid. For a monoid $M$, we let $M^\bullet$ denote the set $M \! \setminus \! \{0\}$, and we let $M^\times$ denote the set of invertible elements of $M$. The monoid $M$ is called \emph{reduced} when $M^\times = \{0\}$. For $x,y \in M$, we say that $x$ \emph{divides} $y$ \emph{in} $M$ and write $x \mid_M y$ provided that there exists $x' \in M$ satisfying $y = x + x'$. An element $a \in M \! \setminus \! M^\times$ is \emph{irreducible} or an \emph{atom} if whenever $a = u + v$ for $u,v \in M$, either $u \in M^\times$ or $v \in M^\times$. The set of atoms of $M$ is denoted by $\mathcal{A}(M)$.

For the remaining of this section, assume that $M$ is a reduced monoid. Let $S$ be a subset of $M$. If no proper submonoid of $M$ contains $S$, then $S$ is called a \emph{set of generators} (or \emph{generating set}) of $M$, in which case we write $M = \langle S \rangle$. The monoid $M$ is called \emph{finitely generated} if $M$ can be generated by a finite set; otherwise, $M$ is called \emph{non-finitely generated}. It is not hard to see that $\mathcal{A}(M)$ is contained in any set of generators of $M$. If $M = \langle \mathcal{A}(M) \rangle$, then $M$ is called \emph{atomic}. 
By contrast, $M$ is called \emph{antimatter} if $\mathcal{A}(M)$ is empty. The notion of being antimatter was introduced and studied in~\cite{CDM99} in the setting of integral domains.

\medskip
\subsection{Factorization Theory}

The (multiplicative) free commutative monoid on $\mathcal{A}(M)$ is denoted by $\mathsf{Z}(M)$ and called \emph{factorization monoid} of $M$; the elements of $\mathsf{Z}(M)$ are called \emph{factorizations}. If $z = a_1 \cdots a_\ell \in \mathsf{Z}(M)$ for some $\ell \in \nn_0$ and $a_1, \dots, a_\ell \in \mathcal{A}(M)$, then $|z| := \ell$ is called the \emph{length} of the factorization $z$. The unique homomorphism $\pi_M \colon \mathsf{Z}(M) \to M$ satisfying that $\pi_M(a) = a$ for all $a \in \mathcal{A}(M)$ is called the \emph{factorization homomorphism} of $M$. For each $x \in M$, the set
\[
	\mathsf{Z}(x) := \pi_M^{-1}(x) \subseteq \mathsf{Z}(M)
\]
is called the \emph{set of factorizations} of $x$. The monoid $M$ is said to be an {\it FF-monoid} (or a \emph{finite factorization monoid}) if $\mathsf{Z}(x)$ is finite for all $x \in M$. It follows from~\cite[Proposition~2.7.8(4)]{GH06} that every finitely generated monoid is an FF-monoid.
For each $x \in M$, the \emph{set of lengths} of $x$ is defined by
\[
	\mathsf{L}(x) := \{|z| : z \in \mathsf{Z}(x)\}.
\]
The set of lengths is an arithmetic invariant of atomic monoids that has been very well studied in recent years (see~\cite{aG16} and the references therein). If $\mathsf{L}(x)$ is a finite set for all $x \in M$, then $M$ is called a \emph{BF-monoid} (or a \emph{bounded factorization monoid}). Clearly, every FF-monoid is a BF-monoid.

\medskip
\subsection{Numerical and Puiseux Monoids}

A special class of atomic monoids is that one comprising all \emph{numerical monoids}, i.e., cofinite submonoids of $(\nn_0,+)$. Each numerical monoid has a unique minimal set of generators, which is finite. Let $N$ be a numerical monoid. If $\{a_1, \dots, a_n\}$ is the minimal set of generators of $N$, then $\mathcal{A}(N) = \{a_1, \dots, a_n\}$ and $\gcd(a_1, \dots, a_n) = 1$. Thus, every numerical monoid is atomic and contains only finitely many atoms. The \emph{Frobenius number} of $N$, denoted by $\mathfrak{f}(N)$, is the minimum $n \in \nn$ such that $\zz_{> n} \subseteq N$. Readers can find an excellent exposition of numerical monoids in~\cite{GR09} and some of their applications in~\cite{AG16}.

A \emph{Puiseux monoid} is a submonoid of $(\qq_{\ge 0},+)$. Clearly, every numerical monoid is a Puiseux monoid. In addition, a Puiseux monoid is isomorphic to a numerical monoid if and only if the former is finitely generated~\cite[Proposition~3.2]{fG17}. Puiseux monoids are not always atomic; for instance, consider $\langle 1/2^n \mid n \in \nn\rangle$. However, if $M$ is a Puiseux monoid such that $0$ is not a limit point of $M^\bullet$, then $M$ is a BF-monoid~\cite[Proposition~4.5]{fG19} and, therefore, atomic. The atomic structure of Puiseux monoids was first studied in~\cite{fG17} and~\cite{GG18}.

\bigskip
\section{Atomicity and Density}
\label{sec:general facts of DPM}

Our goal is to understand how much the set of atoms of an atomic Puiseux monoid can spread through $\rr_{\ge 0}$. To do this it will be convenient to sub-classify Puiseux monoids into classes according to their topological density in positive rays of the real line.

Let $(X, \mathcal{T})$ be a topological space. If $Y \subseteq X$, then $Y$ naturally becomes a topological space with the subspace topology, in which case we write $(Y, \mathcal{T}|_Y)$. Let $A,B \subseteq X$ such that $A \subseteq B$. Recall that $A$ is a dense set of $(X, \mathcal{T})$ (or dense in $X$) if the closure of $A$ is $X$. We say that $A$ is \emph{dense in} $B$ if $A \cap B$ is a dense set of $(B, \mathcal{T}|_B)$. Also recall that $A$ is a nowhere dense of $(X, \mathcal{T})$ (or nowhere dense in $X)$ if the interior of its closure is empty. We say that $A$ is \emph{nowhere dense in} $B$ if $A \cap B$ is a nowhere dense set of $(B, \mathcal{T}|_B)$. Here we only consider the real line $\rr$ with the Euclidean topology.

\begin{definition}
	Let $M$ be a Puiseux monoid. We say that $M$ is
	\begin{enumerate}
		\item \emph{dense} if $M$ is dense in~$\rr_{\ge 0}$;
		\smallskip
		
		\item \emph{eventually dense} if there exists $r \in \rr_{\ge 0}$ such that $M$ is dense in $\rr_{> r}$;
		\smallskip
		
		\item \emph{somewhere dense} if there exists a nonempty open interval $(r,s) \subseteq \rr_{\ge 0}$ such that $M$ is dense in $(r,s)$;
		\smallskip
		
		\item \emph{nowhere dense} if for all $r,s \in \rr_{\ge 0}$ with $r < s$, the set $M$ is not dense in the open interval $(r,s)$.
	\end{enumerate}
\end{definition}

Observe that, \emph{a priori}, none of the definitions above provides information about the set of atoms or any generating set of $M$. However, the density of a Puiseux monoid according to the previous definitions can be characterized in terms of the topological distribution of its set of atoms.

\begin{proposition} \label{prop:characterization of dense PM}
	For an atomic Puiseux monoid $M$ the following conditions are equivalent.
	\begin{enumerate}
		\item $M$ is dense.
		\smallskip
		
		\item $0$ is a limit point of $M^\bullet$.
		\smallskip
		
		\item $0$ is a limit point of any generating set of $M$.
		\smallskip
		
		\item $0$ is a limit point of $\mathcal{A}(M)$.
	\end{enumerate}
\end{proposition}

\begin{proof}
	Clearly, (1) implies (2). Since $M$ is reduced, any generating set of $M$ must contain $\mathcal{A}(M)$ and, therefore, (2), (3), and (4) are equivalent. To prove that any of the conditions (2), (3), and (4) implies~(1), assume that $0$ is a limit point of $M^\bullet$. Let $\{r_n\}$ be a sequence in $M^\bullet$ converging to $0$. Fix $p \in \rr_{> 0}$. To check that $p$ is a limit point of $M$, fix $\epsilon > 0$. Because $\lim_{n \to \infty} r_n = 0$, there exists $n \in \nn$ such that $r_n < \min\{p, \epsilon\}$. Take $m = \max\{k \in \zz \mid p - kr_n > 0 \}$, and set $r = mr_n$. Then we have that $0 < p - r  = p - (m+1) r_n + r_n \le r_n < \epsilon$. As for any $\epsilon > 0$ we have found $r \in M \setminus \{p\}$ with $|p-r|< \epsilon$, it follows that $p$ is a limit point of $M$. So $M$ is a dense Puiseux monoid.
\end{proof}

\begin{cor} \label{cor:density and atomicity}
	If a Puiseux monoid is not dense, then it is atomic.
\end{cor}

\begin{proof}
	It follows immediately from~\cite[Proposition~4.5]{fG19}, which implies that a Puiseux monoid $M$ is a BF-monoid when $0$ is not a limit point of $M^\bullet$. 
\end{proof}

Observe that the conditions~(1), (2), and~(3) in Proposition~\ref{prop:characterization of dense PM} are equivalent even when $M$ is not atomic. In addition, condition~(4) always implies all the three previous conditions. However, the atomicity of $M$ is required to obtain condition~(4) from any of the previous conditions; for example, consider the antimatter Puiseux monoid $\langle 1/2^n \mid n \in \nn \rangle$.
\medskip

Let us characterize the atomic Puiseux monoids that are somewhere dense.

\begin{proposition} \label{prop:PM with somewhere dense generating sets}
	Let $M$ be an atomic Puiseux monoid. Then the following conditions are equivalent.
	\begin{enumerate}
		\item $\mathcal{A}(M)$ is somewhere dense in $\rr_{\ge 0}$.
		\smallskip
		
		\item Each generating set of $M$ is somewhere dense in $\rr_{\ge 0}$.
	\end{enumerate}
	If any of the above conditions holds, then $M$ is a somewhere dense Puiseux monoid.
\end{proposition}

\begin{proof}
	Since $M$ is reduced, every generating set contains $\mathcal{A}(M)$. Therefore~(1) implies~(2). Since $M$ is atomic, $\mathcal{A}(M)$ is a generating set of $M$ and, therefore,~(2) implies~(1). The last statement follows straightforwardly.
\end{proof}

The equivalent conditions in Proposition~\ref{prop:PM with somewhere dense generating sets} are not superfluous as there are examples of atomic Puiseux monoids whose sets of atoms are not only dense in certain interval, but they span to a whole interval.

\begin{example}
	Take $r,s \in \rr_{> 0}$ such that $r < s < 2r$. Now consider the Puiseux monoid $M := \langle (r,s) \cap \qq \rangle$. Since $0$ is not a limit point of $M^\bullet$, it follows that $M$ is atomic. On the other hand, the condition $2r > s$ implies that $s$ is a lower bound for $M^\bullet + M^\bullet$. Hence $\mathcal{A}(M) = (r,s) \cap \qq$.
\end{example}

What is even more striking is the existence of an atomic Puiseux monoid whose set of atoms is dense in $\rr_{\ge 0}$.

\begin{example} \label{ex:atomically dense PM}
	First, we verify that the set $S = \{m/p^n \mid m,n \in \nn \text{ and } p \nmid m\}$ is dense in $\rr_{\ge 0}$ for every $p \in \pp$. To see this, take $\ell \in \rr_{> 0}$ and then fix $\epsilon > 0$. Now take $n,m \in \nn$ with $1/p^n < \epsilon$ and $m/p^n < \ell \le (m+1)/p^n$. Clearly, $s := m/p^n$ of $S$ satisfies that $|\ell - s| < \epsilon$. Since $\epsilon$ was arbitrarily taken, $\ell$ is a limit point of $S$. Because $0$ is also a limit point of $S$, we conclude that $S$ is dense in $\rr_{\ge 0}$.
	
	Now take $\{r_n\}$ to be a sequence of positive rationals with underlying set $R$ dense in $\rr_{\ge 0}$. Let $\{p_k\}$ be an increasing enumeration of the prime numbers. It follows from the previous paragraph, that for each $k \in \nn$, the set
	\[
		\bigg\{\frac m{p_k^n} \ \bigg{|} \ m,n \in \nn \text{ and } p_k \nmid m \bigg\}
	\]
	is dense in $\rr_{\ge 0}$. Therefore, for every natural $k$, there exist naturals $m_k$ and $n_k$ satisfying that $|r_k - m_k/p_k^{n_k}| < 1/k$. Consider the Puiseux monoid
	\begin{equation} \label{eq:PM with dense set of atoms}
		M := \bigg \langle \frac{m_k}{p_k^{n_k}} \ \bigg{|} \ k \in \nn \bigg \rangle.
	\end{equation}
	As distinct generators in \eqref{eq:PM with dense set of atoms} have powers of distinct primes in their denominators, $M$ must be atomic and $\mathcal{A}(M) = \{m_k/p_k^{n_k} \mid k \in \nn \}$. To check that $\mathcal{A}(M)$ is dense in $\rr_{\ge 0}$, take $x \in \rr_{\ge 0}$ and then fix $\epsilon > 0$. Since $R$ is dense in $\rr_{\ge 0}$, there exists $k \in \nn$ large enough such that $1/k < \epsilon/2$ and $|x - r_k| < \epsilon/2$. So $|r_k - m_k/p_k^{n_k}| < 1/k < \epsilon/2$. Then
	\[
		\bigg{|} x - \frac{m_k}{p_k^{n_k}} \bigg{|} < |x - r_k| + \bigg{|} r_k - \frac{m_k}{p_k^{n_k}} \bigg{|} < \epsilon.
	\]
	Hence $\mathcal{A}(M)$ is dense in $\rr_{\ge 0}$.
\end{example}

We conclude this section proving that being somewhere dense and being eventually dense are equivalent conditions in the setting of Puiseux monoids.

\begin{proposition} \label{prop:eventually dense characterization}
	For a Puiseux monoid $M$, the following conditions are equivalent.
	\begin{enumerate}
		\item $M$ is eventually dense.
		\smallskip
		
		\item $M$ is somewhere dense.
	\end{enumerate}
\end{proposition}

\begin{proof}
	It is clear that (1) implies (2). To verify that~(2) implies~(1), suppose that $M$ is somewhere dense, i.e., there exists an interval $(r,s)$ with $r < s$ such that $M$ is topologically dense in $(r,s)$. Since $M$ is closed under addition, it follows that $M$ is topologically dense in $(nr, ns)$ for each $n \in \nn$. Take $N \in \nn$ such that $n(s-r) > r$ for every $n \ge N$. In this case, we can see that $nr < (n+1) r < ns < (n+1)s$ for each $n \ge N$ which means that
	\[
		\rr_{> Nr} = \bigcup_{n=N}^\infty \big( nr, ns \big).
	\]
	Now fix $p \in \rr_{> Nr}$ and $\epsilon > 0$. Take $n \in \nn$ such that $p \in (nr, ns)$. Since $M$ is topologically dense in $(nr,ns)$, there exists $x \in M \cap (nr,ns)$ such that $|x-p| < \epsilon$. Hence $M$ is topological dense in $\rr_{> Nr}$ and, as a result, eventually dense.
\end{proof}

\begin{cor} \label{cor: if the atoms are somewhere dense the monoid is eventually dense}
	Let $M$ be a Puiseux monoid. If $\mathcal{A}(M)$ is somewhere dense, then $M$ is eventually dense.
\end{cor}

The converse of Corollary~\ref{cor: if the atoms are somewhere dense the monoid is eventually dense} does not hold in general, as our next example illustrates. First, recall that the Cantor set $C$ is the subset of $\rr$ one obtains starting with the interval $[0,1]$ and then removing iteratively the open middle third of each of the intervals in each iteration. Formally, we can write
\[
	C := [0,1] \setminus \bigcup _{n=0}^{\infty } \bigcup_{k=0}^{3^n-1} \left( {\frac{3k+1}{3^{n+1}}},{\frac {3k+2}{3^{n+1}}} \right).
\]
It is well known that the Cantor set is an uncountable nowhere dense subset of $[0,1]$ satisfying that $C + C = [0,2]$. Also, it is well known that the set $E$ consisting of the end points of all the intervals obtained in each iteration in the construction of $C$ is a subset of $C$ that is dense in $C$.

\begin{example}\footnote{Looking at endpoints of the Cantor set was kindly suggested by Jyrko Correa and Harold Polo.} \label{ex:atomic and eventually dense PM whose set of atoms is nowhere dense}
	As $C+C = [0,2]$, the equality $(1 + C) + (1 + C) = [2,4]$ holds. Let us argue that $[(1 + C) \cap \qq] + [(1 + C) \cap \qq]$ is dense in the interval $[2,4]$. Note that $1 + C$ is just the Cantor set constructed in the interval $[1,2]$. Since $E \subseteq \qq \cap C$ is dense in $C$, we obtain that $[(1 + C) \cap \qq] + [(1 + C) \cap \qq]$ is dense in $[2,4]$, as desired. Now consider the Puiseux monoid $M = \langle (1 + C) \cap \qq \rangle$. It is clear that $\mathcal{A}(M) = (1 + C) \cap \qq$, which implies that $M$ is atomic. On the other hand, $\mathcal{A}(M)$ is nowhere dense as it is a subset of $1 + C$. However, we have seen before that $M$ is dense in $[2,4]$.
\end{example}

For the sake of completeness, we record the following proposition whose proof is straightforward.

\begin{proposition}
	For a Puiseux monoid $M$, the following statements are equivalent.
	\begin{enumerate}
		\item $M$ is nowhere dense.
		\smallskip
		
		\item Each generating set of $M$ is nowhere dense.
	\end{enumerate}
	If any of the above statements holds, then $M$ is atomic and $\mathcal{A}(M)$ is nowhere dense.
\end{proposition}

Every finitely generated Puiseux monoid is isomorphic to a numerical monoid and, therefore, must be a nowhere dense Puiseux monoids. In addition, there are non-finitely generated atomic Puiseux monoids that are nowhere dense. We shall determine a class of such monoids in Theorem~\ref{thm:increasing PMs}.

\bigskip
\section{Difference Group, Closures, and Conductor}
\label{sec:increasing PM}

\subsection{Difference Group and the Root Closure}

The \emph{difference group} of a monoid $M$, denoted by $\gp(M)$, is the abelian group (unique up to isomorphism) satisfying that any abelian group containing a homomorphic image of $M$ will also contain a homomorphic image of $\gp(M)$. When $M$ is a Puiseux monoid, the difference group $\gp(M)$ can be taken to be a subgroup of $(\qq,+)$, namely,
\[
	\gp(M) = \{ x-y \mid x,y \in M \}.
\]
The difference group of a Puiseux monoid can be characterized in terms of the denominators of its elements. Before stating such a characterization it is convenient to introduce the notion of the root closure. The \emph{root closure} $\widetilde{M}$ of a monoid $M$ with difference group $\gp(M)$ is defined by
\[
	\widetilde{M} := \big\{ x \in \gp(M) \mid nx \in M \ \text{for some} \ n \in \nn \big\}.
\]
The monoid $M$ is called \emph{root-closed} provided that $\widetilde{M} = M$. The difference group and the root closure of a Puiseux monoid $M$ were described by Geroldinger et al. in~\cite{GGT19} in terms of the set of denominators of $M$ in the following way.

\begin{proposition} \cite[Proposition~3.1]{GGT19} \label{prop:closure of a PM}
	Let $M$ be a Puiseux monoid, and let $n = \gcd( \mathsf{n}(M^\bullet))$. Then
	\begin{equation} \label{eq:equality of closures of a PM}
		\widetilde{M} = \gp(M) \cap \qq_{\ge 0} = n \bigg\langle \frac{1}{d} \ \bigg{|} \ d \in \mathsf{d}(M^\bullet) \bigg\rangle.
	\end{equation}
\end{proposition}

The normal closure and the complete integral closure are two other algebraic closures of a monoid that play an important role in semigroup theory and factorization theory. However, we do not formally introduce any of these two notions of closures as, in virtue of \cite[Lemma~2.5]{GR19} and \cite[Proposition~3.1]{GGT19}, they are both equivalent to the root closure in the context of Puiseux monoids.

\begin{example}
	Take $r \in \qq_{> 0}$ such that $\mathsf{d}(r) \in \pp$, and consider the Puiseux monoid $M = \langle r^n \mid n \in \nn_0 \rangle$. Since $1 \in M$, one sees that $\gcd(\mathsf{n}(M^\bullet)) = 1$. Also, it is clear that $\mathsf{d}(M^\bullet) = \{\mathsf{d}(r)^n \mid n \in \nn_0\}$. Then Proposition~\ref{prop:closure of a PM} guarantees that
	\[
		\widetilde{M} = \bigg\langle \frac{1}{\mathsf{d}(r)^n} \ \bigg{|} \ n \in \nn_0 \bigg\rangle,
	\]
	which is the nonnegative cone of the localization $\zz_{\mathsf{d}(r)}$ of $\zz$ at the multiplicative set $\{q^n \mid n \in \nn_0\}$. Observe that $M$ is closed under multiplication, and so it is a cyclic rational semiring. Various factorization invariants of cyclic rational semirings were recently investigated in~\cite{CGG20} by Chapman et al.
\end{example}

Although it is difficult in general to determine whether a given Puiseux monoid is dense, a criterion to determine the density of its difference group and its root closure can be easily established, as the following proposition shows.

\begin{proposition} \label{prop:density of difference group and root closure}
	Let $M$ be a Puiseux monoid. Then the following statements are equivalent.
	\begin{enumerate}
		\item $\gp(M)$ is dense in $\rr$;
		\smallskip
		
		\item $\widetilde{M}$ is dense in $\rr_{\ge 0}$;
		\smallskip
		
		\item $M$ is not finitely generated.
	\end{enumerate}
\end{proposition}

\begin{proof}
	To prove that~(1) implies~(2), suppose that $\gp(M)$ is dense in $\rr$. It is clear that every additive subgroup of $\qq$ is symmetric with respect to $0$. This, together with Proposition~\ref{prop:closure of a PM}, guarantees that $\widetilde{M} = \gp(M) \cap \qq_{\ge 0}$ is dense in $\rr_{\ge 0}$, which is condition~(2).
	\smallskip
	
	Let us argue now that~(2) implies~(3). Suppose, by way of contradiction, that the monoid $M$ is finitely generated. Then $\mathsf{d}(M^\bullet)$ is a finite set, and it follows from Proposition~\ref{prop:closure of a PM} that $m \widetilde{M}$ is a submonoid of $(\nn_0,+)$, where $m = \text{lcm}( \mathsf{d}(M^\bullet))$. Therefore the monoid $m \widetilde{M}$ is not dense in $\rr_{\ge 0}$. However, this clearly implies that $\widetilde{M}$ is not dense in $\rr_{\ge 0}$, which is a contradiction.
	\smallskip
	
	Finally, we show that~(3) implies~(1). Assume that $M$ is not finitely generated. Then $\mathsf{d}(M^\bullet)$ contains infinitely many elements as, otherwise, $\text{lcm}( \mathsf{d}(M^\bullet)) M$ would be a submonoid of $(\nn_0,+)$ and, therefore, finitely generated. Then it follows from Proposition~\ref{prop:closure of a PM} that $0$ is a limit point of $\widetilde{M}^\bullet$, and it follows from Proposition~\ref{prop:characterization of dense PM} that $\widetilde{M}$ is a dense Puiseux monoid. Hence Proposition~\ref{prop:closure of a PM} ensures that $\gp(M) = \widetilde{M} \cup -\widetilde{M}$ is dense in $\rr$.
\end{proof}

\medskip
\subsection{The Conductor of a Puiseux Monoid}

Let $M$ be a monoid. The \emph{conductor} of a Puiseux monoid $M$ is defined to be
\begin{equation} \label{eq:conductor}
	\mathfrak{c}(M) := \{ x \in \gp(M) \mid x + \widetilde{M} \subseteq M \}.
\end{equation}
It is clear that $\mathfrak{c}(M)$ is a subsemigroup of the group $\gp(M)$. Although it is more convenient for our purposes to define the conductor of a Puiseux monoid in terms of its root closure, we would like to remark that in general the conductor of a monoid is defined in terms of its complete integral closure; see for example, \cite[Definition~2.3.1]{GH06}. However, recall that the notions of root closure and complete integral closure coincide in the setting of Puiseux monoids.

\begin{example} \label{ex:conductor of NMs}
	Let $M$ be a numerical monoid, and let $\mathfrak{f}(M)$ be the Frobenius number of $M$. It follows from~Proposition~\ref{prop:closure of a PM} that $\gp(M) = \zz$ and $\widetilde{M} = \nn_0$. For $n \in M$ with $n \ge \mathfrak{f}(M) + 1$, the inclusion $n + \widetilde{M} = n + \nn_0 \subseteq M$ holds . In addition, for each $n \in \zz$ with $n \le \mathfrak{f}(M)$ the fact that $\mathfrak{f}(M) \in n + \widetilde{M}$ implies that $n + \widetilde{M} \nsubseteq M$. Thus,
	\begin{equation} \label{eq:Frobenius number and conductor}
		\mathfrak{c}(M) = \{ n \in \zz \mid n \ge \mathfrak{f}(M) + 1 \}.
	\end{equation}
	As the equality of sets~(\ref{eq:Frobenius number and conductor}) shows, the minimum of $\mathfrak{c}(M)$ is $\mathfrak{f}(M) + 1$, namely, the conductor number of $M$ as usually defined in the setting of numerical monoids.
\end{example}

The conductor of a Puiseux monoid has been recently described in~\cite{GGT19} as follows.

\begin{proposition} \cite[Proposition~3.2]{GGT19} \label{prop:conductor of a PM} 
	Let $M$ be a Puiseux monoid. Then the following statements hold.
	\begin{enumerate}
		\item If $M$ is root-closed, then $\mathfrak{c}(M) = \widetilde{M} = M$.
		\smallskip
		
		\item If $M$ is not root-closed, then set $\sigma = \sup \, \widetilde{M} \setminus M$.
		\begin{enumerate} \label{part 2: conductor of PM}
			\item If $\sigma = \infty$, then $\mathfrak{c}(M) = \emptyset$.
			\smallskip
			
			\item If $\sigma < \infty$, then $\mathfrak{c}(M) = M_{\ge \sigma}$.
		\end{enumerate}
	\end{enumerate}
\end{proposition}

Puiseux monoids with nonempty conductor have been considered in~\cite{GGT19} and, more recently, in~\cite{BG20}. The density of a Puiseux monoid with nonempty conductor can be fully understood with the following criterion.

\begin{proposition}
	Let $M$ be a Puiseux monoid with nonempty conductor. Then $M$ is nowhere dense if and only if $M$ is finitely generated.
\end{proposition}

\begin{proof}
	For the direct implication, suppose that $M$ is a nowhere dense Puiseux monoid. Since $M$ has nonempty conductor, it follows from Proposition~\ref{prop:conductor of a PM} that either $M$ is root-closed or $\sup \widetilde{M} \setminus M < \infty$. We consider the following two cases.
	\smallskip
	
	\noindent CASE 1: $M$ is root-closed. Suppose, by way of contradiction, that $M$ is not finitely generated. Then Proposition~\ref{prop:density of difference group and root closure} guarantees that $\widetilde{M}$ is dense in $\rr_{\ge 0}$. Since $M$ is root-closed $M = \widetilde{M}$, and so $M$ is dense in $\rr_{\ge 0}$. However, this contradicts that $M$ is nowhere dense.
	\smallskip
	
	\noindent CASE 2: $M$ is not root-closed. In this case, the inequality $\sup \widetilde{M} \setminus M < \infty$ must hold. Taking $\tau := 1 + \sup \widetilde{M} \setminus M$, we find that $M_{\ge \tau} = \widetilde{M}_{\ge \tau}$. Suppose for a contradiction that $M$ is not finitely generated. Then $\widetilde{M}_{\ge \tau}$ would be dense in $\rr_{\ge \tau}$ by Proposition~\ref{prop:density of difference group and root closure} and, therefore, $M_{\ge \tau}$ would also be dense in $\rr_{\ge \tau}$. However, this contradicts that $M$ is nowhere dense.
	\smallskip
	
	The reverse implication follows immediately. Indeed, if $M$ is finitely generated, then $\mathsf{d}(M^\bullet)$ is finite and the set $\text{lcm}(\mathsf{d}(M^\bullet)) M \subseteq \nn$ is nowhere dense in $\rr_{\ge 0}$, which in turns implies that $M$ is nowhere dense.
\end{proof}

\medskip
\subsection{Increasing Puiseux Monoids}

A submonoid $M$ of $(\rr_{\ge 0})$ is called \emph{increasing} if $M$ can be generated by an increasing sequence of real numbers. Clearly, every Puiseux monoid generated by an increasing sequence of rationals is an example of an increasing monoid. Increasing (and decreasing) Puiseux monoids were first studied in~\cite{GG18}. If an increasing submonoid of $\rr_{\ge 0}$ (in particular, an increasing Puiseux monoid) can be generated by an unbounded (resp., bounded) sequence it is called a \emph{strongly increasing} (resp., \emph{weakly increasing}) monoid. Strongly increasing monoids were considered in~\cite{mBA19,fG19} and more recently in~\cite{mBA20} under the term ``$\omega$-monoids." Increasing Puiseux monoids are always atomic. Indeed, it was proved in~\cite[Proposition~4.5]{fG19} that every increasing Puiseux monoid is an FF-monoid.

Our goal in this final subsection is to prove that increasing Puiseux monoids are nowhere dense and also that they have empty conductors provided that they are not finitely generated. First, we will establish some needed lemmas.
\smallskip

Let $S$ be a subset of $\rr$, and take $\alpha \in \rr$. We say that $\alpha$ is a limit point of $S$ \emph{from the right} if for every $\epsilon \in \rr_{> 0}$ the set $(\alpha, \alpha+ \epsilon) \cap S$ is nonempty. 

\begin{lemma}\label{lem:contradictone}
	Let $M$ be an increasing Puiseux monoid. If $M$ has a limit point $\alpha$ from the right, then $M$ also has another limit point $\beta$ from the right with $\beta < \alpha$.
\end{lemma}

\begin{proof}
	Assume that $M$ has a limit point $\alpha$ from the right. Then $M$ cannot be finitely generated. Since $M$ is increasing it is an atomic monoid. As $M$ is not finitely generated there exists a strictly increasing sequence $(a_n)_{n \in \nn}$ with underlying set $\mathcal{A}(M)$. Notice that if the equality $\lim_{n \to \infty} a_n = \infty$ held, it would imply that $|M \cap [0,n]| < \infty$ for every $n \in \nn$, contradicting that $\alpha$ is a limit point of $M$ from the right. Hence the sequence $(a_n)_{n \in \nn}$ must converge to some $\ell \in \rr_{> 0}$.
	
	For each $n \in \nn$ the monoid $\langle a_1, \dots, a_n \rangle$ is finitely generated and, therefore,  the set $\{ s \in \langle a_1, \dots, a_n \rangle \mid s > \alpha \}$ has a minimum, which we denote by $s_n$. Clearly, the sequence $(s_n)_{n \in \nn}$ is decreasing. Since $\alpha$ is a limit point of $M$ from the right, $\lim_{n \to \infty} s_n = \alpha$. Let $(j_n)_{n \in \nn}$ be the strictly increasing sequence with underlying set $\{j \in \nn \mid s_j < s_{j-1}\}$. Notice that $a_{j_n} \mid_M s_{j_n}$ for every $n \in \nn$. Hence after setting $b_n := s_{j_n} - a_{j_n}$ for every $n \in \nn$, we obtain that each term of the sequence $(b_n)_{n \in \nn}$ belongs to $M$. In addition, $b_n = s_{j_n} - a_{j_n} > s_{j_{n+1}} - a_{j_{n+1}} = b_{n+1}$ for every $n \in \nn$, whence $(b_n)_{n \in \nn}$ is a strictly decreasing sequence. As $(b_n)_{n \in \nn}$ consists of nonnegative numbers, $\lim_{n \to \infty} b_n = \beta$ for some $\beta \in \rr_{\ge 0}$. So $\beta$ is a limit point of $M$ from the right. It is clear that $\beta = \lim_{n \to \infty} b_n  = \lim_{n \to \infty} s_{j_n} - \lim_{n \to \infty} a_{j_n} = \alpha - \ell < \alpha$, from which the lemma follows.
\end{proof}

\begin{lemma} \label{lem:increasing PMs have no limit points from the right}
	Let $M$ be an increasing Puiseux monoid. Then $M$ does not contain limit points from the right.
\end{lemma}

\begin{proof}
	Suppose, by way of contradiction, that the set $A$ consisting of all the limit points of $M$ from the right is nonempty. Set $\alpha = \inf A$. We will first argue that $\alpha$ is indeed the minimum of $A$. To do so take a decreasing sequence $(\alpha_n)_{n \in \nn}$ whose terms belong to $A$ such that $\lim_{n \to \infty} \alpha_n = \alpha$. Because each $\alpha_n$ belongs to $A$, for each fixed $n \in \nn$ there exists a strictly decreasing sequence $(b_{n,j})_{j \in \nn}$ whose terms belong to $M$ such that $\lim_{j \to \infty} b_{n,j} = \alpha_n$. So for each $n \in \nn$ there exists $k_n \in \nn$ such that $b_{n, k_n} < \alpha_n + 1/n$. Therefore $(b_{n,k_n})_{n \in \nn}$ is a sequence of elements of $M$ satisfying that $\lim_{n \to \infty} (b_{n,k_n}) = \alpha$. Hence $\alpha$ is a limit point of $M$ from the right and, therefore, $\alpha \in A$. So $\alpha = \min A$. Now Lemma~\ref{lem:contradictone} guarantees the existence of a limit point $\beta$ of $M$ from the right such that $\beta < \alpha$. However, this contradicts the minimality of $\alpha$. Thus, one can conclude that $M$ has no limit point from the right.
\end{proof}

We are now ready to prove the main result of this section. 

\begin{theorem} \label{thm:increasing PMs}
	Let $M$ be an increasing Puiseux monoid. Then the following statements hold.
	\begin{enumerate}
		\item $M$ is nowhere dense.
		\smallskip
		
		\item $\mathfrak{c}(M)$ is nonempty if and only if $M$ is finitely generated.
	\end{enumerate}
\end{theorem}

\begin{proof}
	To argue the statement~(1) we proceed by contradiction. Suppose that $M$ is not a nowhere dense Puiseux monoid. Then there exist $r,s \in \rr$ with $0 < r < s$ such that $M$ is dense in $(r,s)$. Take $b \in M \cap (r,s)$. Since $M$ is an increasing Puiseux monoid, Lemma~\ref{lem:increasing PMs have no limit points from the right} guarantees the existence of $\epsilon \in \rr_{> 0}$ such that $M \cap [b, b + \epsilon] = \{ b \}$. Then the set $M \cap (b, \min\{s, b + \epsilon\})$ is empty. However, this contradicts that $M$ is dense in $(r,s)$ because the open interval $(b, \min\{s, b + \epsilon\})$ is contained in the open interval $(r,s)$. Hence $M$ is a nowhere dense Puiseux monoid, as desired.
	\smallskip
	
	To establish the direct implication of~(2), suppose that $M$ be a non-finitely generated increasing Puiseux monoid. Then, by Proposition~\ref{prop:density of difference group and root closure}, the set $\widetilde{M}^\bullet$ has $0$ as a limit point. Since $0$ is not a limit point of $M^\bullet$, it follows that $M \ne \widetilde M$. Hence $M$ cannot be root-closed. On the other hand, let us argue that $\widetilde M \setminus M$ is unbounded and so $\sup \, \widetilde M \setminus M = \infty$. Indeed, let us see that for each $r \in \rr$ there exists an element $\widetilde c$ in $\widetilde M \setminus M$ with $\widetilde c > r$. Since the underlying set of $M$ is unbounded, there exists $c \in M$ such that $c > r$. It follows now from Lemma~\ref{lem:increasing PMs have no limit points from the right} that there exists $\epsilon \in \rr_{> 0}$ such that $M \cap [c, c + \epsilon] = \{c\}$. Since $M$ is not finitely generated, Proposition~\ref{prop:density of difference group and root closure} ensures that $\widetilde{M}$ is dense in $\rr_{\ge 0}$ and, therefore, there exists $\widetilde c_0 \in \widetilde M$ such that $0 < \widetilde c_0 < \epsilon$. Now take $\widetilde c := c + \widetilde c_0$. The element $\widetilde c$ belongs to $\widetilde M$ and satisfies that $c < \widetilde c < c + \epsilon$, whence $\widetilde c \in \widetilde M \setminus M$. In addition, $\widetilde{c} > c > r$. As $r$ was taken arbitrarily in $\rr$, we obtain that $\sup \, \widetilde{M} \setminus M = \infty$, as desired. Hence it follows from Proposition~\ref{prop:conductor of a PM} that $\mathfrak{c}(M)$ is necessarily empty.
	
	For the reverse implication of~(2), it suffices to notice that $M$ is finitely generated if and only if $M$ is isomorphic to a numerical monoid and that, as seen in Example~\ref{ex:conductor of NMs}, the conductor of a numerical monoid is always nonempty.
\end{proof}

We have seen that if a non-finitely generated monoid is increasing, then it has empty conductor. We would like to remark that there are non-finitely generated atomic Puiseux monoids that are not increasing and still have empty conductor. The following example illustrates this.

\begin{example}
	Consider the Puiseux monoid $M$ generated by the infinite set $A := \{1\} \cup \{1 + 1/p \mid p \in \pp\}$. Since $0$ is not a limit point of $M^\bullet$, it follows from \cite[Proposition~4.5]{fG19} that $M$ is atomic. Indeed, it is not hard to check that $\mathcal{A}(M) = A$. So $M$ is not finitely generated. On the other hand, the fact that $1$ is a limit point of $M$ from the right, along with Lemma~\ref{lem:increasing PMs have no limit points from the right}, guarantees that $M$ is not an increasing Puiseux monoid. In addition, it has been verified in \cite[Example~3.9]{GGT19} that $\mathfrak{c}(M)$ is empty.
\end{example}

\bigskip
\section*{Acknowledgments}

The authors would like to thank Felix Gotti for his valuable feedback during the preparation of this paper.
\medskip

\bigskip

\end{document}